\documentclass[12pt]{amsart}
\usepackage{amsmath}
\usepackage{amsthm,amsxtra,amssymb,amsfonts,latexsym, mathrsfs, dsfont, bbm, bm, verbatim, mathtools}

\usepackage[initials,alphabetic]{amsrefs}
\usepackage[usenames]{color}
\usepackage[all]{xy}
\usepackage{hyperref}
\renewcommand{\vec}[1]{\underline{#1}}

\allowdisplaybreaks[1]

\newtheorem{lemma}{Lemma}
\newtheorem{prop}[lemma]{Proposition}
\newtheorem{theorem}[lemma]{Theorem}
\newtheorem{cor}[lemma]{Corollary}

\theoremstyle{remark}

\theoremstyle{definition}

\newcommand{\lge}{\langle}
\newcommand{\rge}{\rangle}

\newcommand{\ax}{\mathcal{A}}
\newcommand{\bx}{\mathcal{B}}

\newcommand{\fx}{\mathcal{F}}

\newcommand{\nx}{\mathcal{N}}

\newcommand{\nz}{\mathbb{N}}
\newcommand{\rz}{\mathbb{R}}
\newcommand{\cz}{\mathbb{C}}

\newcommand{\fz}{\mathbb{F}}

\newcommand{\Ga}{\Gamma}
\newcommand{\Om}{\Omega}

\newcommand{\de}{\delta}
\newcommand{\si}{\sigma}

\newcommand{\eps}{\varepsilon}
\newcommand{\8}{\infty}

\mathchardef\dash="2D

\newcommand{\Dom}{\operatorname{Dom}}

\begin{document}
\title{An application of free transport to mixed $q$-Gaussian algebras}
\date{}
\author[Brent Nelson]{Brent Nelson$^\bullet$}
\address{$\bullet$ Department of Mathematics, University of California, Berkeley, CA 94709}
\email{brent@math.berkeley.edu}
\thanks{$\bullet$ Research supported by the NSF awards DMS-1161411 and DMS-1502822.}
\author[Qiang Zeng]{Qiang Zeng$^\circ$}
\address{$\circ$ Center of Mathematical Sciences and Applications, Harvard University, Cambridge, MA 02138}
\email{qzeng@cmsa.fas.harvard.edu}
\maketitle

\begin{abstract}
We consider the mixed $q$-Gaussian algebras introduced by Speicher which are generated by the variables $X_i=l_i+l_i^*,i=1,\ldots,N$, where $l_i^* l_j-q_{ij}l_j l_i^*=\de_{i,j}$ and $-1<q_{ij}=q_{ji}<1$. Using the free monotone transport theorem of Guionnet and Shlyakhtenko, we show that the mixed $q$-Gaussian von Neumann algebras are isomorphic to the free group von Neumann algebra $L(\fz_N)$, provided that $\max_{i,j}|q_{ij}|$ is small enough. The proof relies on some estimates which are generalizations of Dabrowski's results for the special case $q_{ij}\equiv q$.
\end{abstract}

\section{Introduction}
A fundamental problem in the theory of operator algebras is whether two algebras are isomorphic. The operator algebra (both the (reduced) $C^*$-algebra and von Neumann algebra) of the free group $\fz_N$ with $N$ generators has been a central object to study. In particular, the von Neumann algebras of $\fz_N$ are isomorphic to those generated by $N$ free semi-circular variables $(S_i)_{i=1,...,N}$ due to Voiculescu; see \cite{VDN}. Motivated from mathematical physics, Bo\.zejko and Speicher introduced the $q$-Gaussian variables \cite{BS91}, which can be regarded as a deformation of the free semi-circular system. Since then, the $q$-Gaussian algebras have been extensively studied. For an incomplete list of results, see \cites{BKS, Sh04, Nou,Sn04,Ri05,KN11,Av} among others. More recently, using estimates of Dabrowski \cite{Dab}, Guionnet and Shlyakhtenko \cite{GS14} have shown that the $q$-Gaussian von Neumann algebras are isomorphic to those generated from the free semi-circular variables for $|q|$ small enough. This result was proved using the powerful free monotone transport theorem. The first named author \cite{Ne15} adapted this to the non-tracial setting and showed that the finitely generated $q$-deformed free Araki-Woods algebras are isomorphic to the finitely generated free Araki-Woods factor for $|q|$ small enough (\emph{cf.} \cite{Shl97}, \cite{Hia03}). In this paper, we give another application of Guionnet and Shlyakhtenko's theory.

The $q$-Gaussian variables and $q$-commutation relations were further generalized with the motivation from physics. In \cite{Sp93}, Speicher introduced the commutation relation
\begin{equation}\label{qcom}
l_i^* l_j-q_{ij}l_j l_i^*=\de_{i,j}
\end{equation}
where $Q=(q_{ij})_{i,j=1}^N$ is a symmetric matrix with $|q_{ij}|\le 1$, and $\de_{i,j}$ is the Kronecker delta function. It was shown \cites{Sp93,BS94} that \eqref{qcom} can be represented as left creation and annihilation operators on a certain Fock space. We will always use this Fock representation of \eqref{qcom} in this paper. We call the operator algebras generated by $X_i=l_i+l_i^*$ the mixed $q$-Gaussian algebras and call $X_i$'s the mixed $q$-Gaussian variables. In fact, the so-called braid relations (a.k.a.\! Yang--Baxter equation), which are more general than \eqref{qcom}, were also studied by Bo\.zejko, Speicher, Nou, and Kr\.olak in \cite{BS94, Nou, Kr00,Kr05}, among others. As for \eqref{qcom}, Lust-Piquard \cite{LP99} showed the $L^p$ boundedness of the Riesz transforms associated to the number operator of the system. More recently, Junge and the second named author \cite{JZ15} studied various properties of the mixed $q$-Gaussian von Neumann algebras and in particular proved that they have the complete metric approximation property and are strongly solid in the sense of Ozawa and Popa \cite{OP10} as long as $\max_{1\le i,j\le N} |q_{ij}|<1$.

In the present paper, we show that if $\max_{1\le i,j\le N} |q_{ij}|$ is small enough then the mixed $q$-Gaussian algebras are isomorphic to the algebras generated from free semi-circular variables. To state the result precisely, let us denote by $\Ga_q(\rz^N)$ the $q$-Gaussian von Neumann algebra of $N$ generators, $L(\fz_N)$ the von Neumann algebra generated from $\fz_N$, and $C^*(Y_1,...,Y_N)$ the $C^*$-algebra generated by operators $Y_1,...,Y_N$.
\begin{theorem}\label{main}
Let $Q=(q_{ij})$ be a symmetric $N\times N$ matrix with $N\in \{2,3,\ldots\}$ and $q_{ij}\in(-1,1)$. Let $\Ga_Q$ be the von Neumann algebra generated by the mixed $q$-Gaussian variables $X_1,\ldots,X_N$. Then there exists a $q_0=q_0(N)>0$ depending only on $N$ such that $\Ga_Q\cong\Ga_0(\rz^N)\cong L(\fz_N)$ and $C^*(X_1,\ldots ,X_N)\cong C^*(S_1,...,S_N)$ for all $Q$ satisfying $\max_{i,j} |q_{ij}|<q_0$. 
\end{theorem}
The proof of this theorem relies on the construction of the conjugate variables and potentials for $\Ga_Q$. To this end, we follow the idea of Dabrowski \cite{Dab} and obtain some estimates which are generalized from similar ones for the $q_{ij}\equiv q$ case.


\section{The Mixed $q$-Gaussian Algebra}

We refer the readers to \cites{BS94,LP99,JZ15} for unexplained preliminary facts for the mixed $q$-Gaussian variables. Let $(e_i)_{i=1}^N$ be an orthonormal basis of $\rz^N$. The Fock space associated with the mixed $q$-Gaussian variables is defined as $\fx_Q=\oplus_{n=0}^\8 H_Q^n$, where $H_Q^n$ is isomorphic to $(\cz^N)^{\otimes n}$ as a vector space and $H_Q^0=\cz \Om$ with $\Om$ being the vacuum state. Let $S_n$ denote the symmetric group on $n$ elements and write $\vec{i}=(i_1,\ldots,i_n)$ for a vector in $[N]^n:=\{1,\ldots,N\}^n$. The inner product of $\fx_Q$ is given by
\[
\lge e_{i_1}\otimes\cdots\otimes e_{i_m}, e_{j_1}\otimes\cdots\otimes e_{j_n} \rge_Q = \de_{m,n} \sum_{\si\in S_n} a(\si,\vec{j}) \lge e_{i_1}, e_{j_{\si^{-1}(1)}}\rge \cdots \lge e_{i_m}, e_{j_{\si^{-1}(n)}}\rge.
\]
Here $a(\si,\vec{j})$ is a product of $(q_{kl})$ defined as follows: We write $\tau_{1}=(12), \tau_2=(23),\ldots,\tau_n=(n1)$ for transpositions. It is well known that $(\tau_i)_{i=1}^n$ is a generating set of $S_n$ and that the number of inversions of $\si\in S_n$ is given by
\[
|\si| = \min\{k\in \nz: \si=\tau_{i_1}\cdots\tau_{i_k}\}.
\]
For $\si\in S_n$, assume $|\si|=k$ and $\si=\tau_{m_1}\cdots\tau_{m_k}$. Then (see \cites{BS94,LP99})
\[
a(\si, \vec{i}) = \prod_{j=1}^{k-1} q(i_{\si_j(m_{k-j})},i_{\si_j(m_{k-j}+1)})q(i_{m_k}, i_{m_k +1}),
\]
where $\si_j=\tau_{m_{k-j+1}}\cdots\tau_{m_k}$ and we have written $q_{i_1i_2}=q(i_1,i_2)$. By definition,
\[
l_i(e_{j_1}\otimes \cdots\otimes e_{j_n}) = e_i\otimes e_{j_1}\otimes \cdots\otimes e_{j_n},
\]
\[
r_i(e_{j_1}\otimes \cdots\otimes e_{j_n})=e_{j_1}\otimes \cdots\otimes e_{j_n}\otimes e_i,
\]
\[
l_i^*(e_{j_1}\otimes \cdots\otimes e_{j_n})=\sum_{k=1}^n \de_{i,j_k} q_{ij_1}\cdots q_{ij_{k-1}} e_{j_1}\otimes \cdots\otimes e_{j_{k-1}}\otimes e_{j_{k+1}} \otimes\cdots\otimes e_{j_n}.
\]
Here $l_i=l(e_i)$ is the left creation operator and $l_i^*$ the left annihilation operator. One can check that $l_i^*$ is the adjoint operator of $l_i$ with respect to the inner product $\lge\cdot,\cdot \rge_Q$ of $L^2(\Ga_Q,\tau_Q)$. Similarly, $r_i$  and $r_i^*$ are the right creation and annihilation operator, respectively. Let $X_i= l_i + l_i^*$ be the mixed $q$-Gaussian variables. Let $\Ga_Q$ denote the mixed $q$-Gaussian von Neumann algebra generated by $X_i, i=1,\ldots,N$. By \cite{BS94}, there is a normal faithful tracial state $\tau_Q$ on $\Ga_Q$ defined as $\tau_Q(X)=\lge X\Om, \Om\rge_Q$ for $X\in \Ga_Q$. If $\max_{ij} |q_{ij}|<1$, then there is a canonical unitary isomorphism between $L^2(\Ga_Q,\tau_Q)$ and $\fx_Q$ given by
\[
X\mapsto X\Om, \text{ for } X\in \Ga_Q,
\]
which extends continuously to $L^2(\Ga_Q)$. From time to time this identification will be used implicitly in the following and we write $\lge\cdot,\cdot\rge_{\tau_Q}$ for the inner product of $L^2(\Ga_Q,\tau_Q)$. Given a finite-length tensor $\xi\in \fx_Q$, there is a unique element $W(\xi)$ in $\Ga_Q$ such that $W(\xi)\Om= \xi$, and $W(e_{i_1}\otimes\cdots\otimes e_{i_n})$ is called the Wick word (a.k.a.\!  Wick product in the literature) of $e_{i_1}\otimes\cdots\otimes e_{i_n}$.

Following \cites{GS14,Dab}, we consider $\cz\lge Y_1,\ldots,Y_N\rge$, the algebra of noncommutative polynomials in $N$ self-adjoint variables. Given a noncommutative power series
$$F(Y_1,\ldots,Y_N)=\sum_{\vec{i},p} a_{\vec{i},p} Y_{i_1}\cdots Y_{i_p}\otimes Y_{i_{p+1}}\cdots Y_{i_n}$$
whose radius of convergence is greater than $R>1$, we define the norm $\|F\|_R = \sum_{\vec{i},p} |a_{\vec{i},p}| R^n$. Similarly, for
	\[
		F(Y_1,\ldots, Y_N)=\sum_{\vec{i}} a_{\vec{i}} Y_{i_1}\cdots Y_{i_n},
	\]
with radius of convergence greater than $R>1$ we define $\|F\|_R=\sum_{\vec{i}} |a_{\vec{i}}| R^n$. For an algebra $\ax$, we write $\ax^{op}$ for the opposite algebra of $\ax$, and write $a^\circ\in \ax^{op}$ whenever $a\in \ax$.

\section{The Derivation $\partial_j^{(Q)}$ and $\Xi_j$}

Consider the linear map
\[
 \partial_j^{(Q)}: \cz\lge X_1,\ldots,X_N\rge \to \bx(L^2(\Ga_Q)), \quad \partial_j^{(Q)}(X) = [X, r_j]:= Xr_j-r_jX.
\]
For $i=1,\ldots,N$, define
\[
\Xi_i: \fx_Q\to \fx_Q, \quad \Xi_i (e_{j_1}\otimes \cdots\otimes e_{j_n}) = q_{ij_1}\cdots q_{ij_n}e_{j_1} \otimes \cdots\otimes e_{j_n}.
\]
We also write $q_i(\vec{j})=q_{ij_1}\cdots q_{ij_n}$ for short.

For each $n\geq 1$, we consider the following equivalence relation on $[N]^n$: $\vec{i}\sim\vec{j}$ if $\exists \sigma\in S_n$ such that
	\begin{align*}
		\vec{i}=\sigma\cdot\vec{j}=(j_{\sigma(1)},\ldots, j_{\sigma(n)}).
	\end{align*}
Let $[\vec{i}]$ denote the equivalence class of $\vec{i}\in [N]^n$. Note that $q_k(\vec{j})=q_k(\vec{i})$ for each $\vec{j}\in [\vec{i}]$ and each $k=1,\ldots, N$; consequently, we may at times denote $q_k(\vec{i})$ by $q_k([\vec{i}])$. For each equivalence class $[\vec{i}]$ we define the subspace
	\begin{align*}
		\fx_{[\vec{i}]}:=\text{span}\left\{ e_{j_1}\otimes\cdots\otimes e_{j_n}\colon \vec{j}\in [\vec{i}]\right\},
	\end{align*}
and denote by $p_{[\vec{i}]}$ the orthogonal projection onto $\fx_{[\vec{i}]}$. It is easy to see that $H_Q^0$ along with the subspaces $\fx_{[\vec{i}]}$  (ranging over all equivalence classes and all $n\geq 1$) offers an orthogonal decomposition of $\fx_Q$, and consequently
	\begin{align*}
		p_\Omega + \sum_{n\geq 1} \sum_{[\vec{i}]\in [N]^n/\sim} p_{[\vec{i}]} = 1,
	\end{align*}
where $p_\Omega$ is the projection onto the vacuum vector. For notational consistency, we will often denote $p_\Omega=p_{[(\emptyset)]} \in [N]^0/\sim$.

For each $j=1,\ldots, N$ it follows that
	\begin{align}\label{HS}
		\Xi_j =\sum_{n\geq 0} \sum_{[\vec{i}]\in [N]^n/\sim} q_j(\vec{i})p_{[\vec{i}]}.
	\end{align}
Since $q_j(\vec{i})$ are real numbers, $\Xi_j$ is a self-adjoint operator. Moreover, if $q:=\max_{1\leq i,j\leq N}|q_{ij}|$ satisfies $q^2N<1$ then $\Xi_j\in HS(\fx_Q)$, the Hilbert--Schmidt operators on $\fx_Q$, since for each $n\geq 1$
	\begin{align*}
		\sum_{[\vec{i}]\in [N]^n/\sim} \|p_{[\vec{i}]}\|_{HS}^2 =\sum_{k_1+\cdots +k_N=n} \binom{n}{k_1,\ldots,k_N}= N^n.
	\end{align*}

Noting that $[l_i,r_j]=0$, we see that
	\[
		{\partial}^{(Q)}_j(X_i)(e_{i_1}\otimes\cdots\otimes e_{i_n}) = \de_{i,j} q_{ii_1}\cdots q_{ii_n} e_{i_1}\otimes\cdots\otimes e_{i_n},
	\]
and hence ${\partial}^{(Q)}_j(X_i)=\de_{i,j} \Xi_j$. As the space of Hilbert--Schmidt operators is a two-sided ideal in $\bx(\fx_Q)$, the Leibniz rule implies ${\partial}^{(Q)}_j$ maps $\cz\lge X_1,\ldots, X_N\rge$ into $HS(\fx_Q)$ for each $j=1,\ldots, N$ whenever $\Xi_j\in HS(\fx_Q)$. When this is the case, we think of $\partial_j^{(Q)}$ as a densely defined derivation
	\[
		\partial_j^{(Q)}\colon L^2(\Ga_Q,\tau_Q)\to HS(\fx_Q).
	\]
Recall that $L^2(\Ga_Q\bar{\otimes}\Ga_Q^{op},\tau_Q\otimes \tau_Q^{op})$ is isomorphic to $HS(\fx_Q)$ via the map
	\begin{align*}
		a\otimes b^\circ\mapsto \lge\cdot, b^*\Omega\rge a\Omega.
	\end{align*}
In particular, $1\otimes 1^\circ \mapsto p_\Omega$. We will usually think of $\partial_j^{(Q)}$ as having range $L^2(\Ga_Q\bar{\otimes}\Ga_Q^{op},\tau_Q\otimes\tau_Q^{op})$.

\begin{prop}\label{deriv}
Suppose $\Xi_j\in HS(\fx_Q)$. Then $\partial_j^{(Q)*}(1\otimes 1^\circ)=X_j$.
\end{prop}
\begin{proof}
Fix $\vec{i}\in [N]^n$ and let $\pi_1\in \mathcal{B}(\mathcal{F}_Q)$ denote the projection onto tensors of length one. Then there exist scalars $c_1,\ldots, c_n$ such that
	\[
		\pi_1 X_{i_1}\cdots X_{i_n}\Omega = \sum_{t=1}^n c_t  e_{i_t},
	\]
where we are summing over which operator $X_{i_1},\ldots, X_{i_n}$ created the vector $e_{i_t}$. We claim
	\begin{align*}
		c_t &= \sum_{d\geq 0} \sum_{[\vec{j}]\in [N]^d/\sim} q_{i_t}(\vec{j})\lge X_{i_1}\cdots X_{i_{t-1}}p_{[\vec{j}]} X_{i_{t+1}}\cdots X_{i_n}\Omega, \Omega\rge_Q\\
		&= \lge X_{i_1}\cdots X_{i_{t-1}} \Xi_{i_t} X_{i_{t+1}}\cdots X_{i_n}\Omega, \Omega\rge_Q.
	\end{align*}
First note that the second equality is immediate from (\ref{HS}). Now, the only terms from $\pi_1 X_{i_1}\cdots X_{i_n}\Omega$ which contribute to $c_t$ are those where $X_{i_t}$ creates $e_{i_t}$; that is, ones where the creation operator rather than the annihilation operator in $X_{i_t}$ acts. Hence towards computing $c_t$ we may replace $X_{i_t}$ with $l_{i_t}$ and compute
	\[
		\pi_1 X_{i_1}\cdots X_{i_{t-1}} l_{i_t} X_{i_{t+1}}\cdots X_{i_n} \Omega.
	\]
Recall that we have the partition of unity $\{p_{[\vec{j}]}\colon d\geq 0,\ [\vec{j}]\in [N]^d/\sim\}$. For each $d\geq 0$ and $[\vec{j}]\in [N]^d/\sim$, let $\{\zeta^{[\vec{j}]}_\ell\}$ be an orthonormal basis for $\fx_{[\vec{j}]}$. Then we have
	\begin{align*}
		\pi_1 X_{i_1}\cdots X_{i_{t-1}} &l_{i_t} X_{i_{t+1}}\cdots X_{i_n} \Omega\\
			&=\sum_{d\geq 0} \sum_{[\vec{j}]\in [N]^d/\sim} \pi_1 X_{i_1}\cdots X_{i_{t-1}} l_{i_t}  p_{[\vec{j}]} X_{i_{t+1}}\cdots X_{i_n} \Omega\\
			&=\sum_{d\geq 0} \sum_{[\vec{j}]\in [N]^d/\sim}\sum_{\ell} \pi_1 X_{i_1}\cdots X_{i_{t-1}} e_{i_t}\otimes \zeta^{[\vec{j}]}_\ell \left\lge X_{i_{t+1}}\cdots X_{i_n}\Omega, \zeta^{[\vec{j}]}_\ell\right\rge_Q.
	\end{align*}
Furthermore, of the above terms the only ones which contribute to $c_t$ are those where $e_{i_t}$ survives; that is, where none of the operators $X_{i_1},\ldots, X_{i_{t-1}}$ annihilate $e_{i_t}$. And yet, to survive the action of $\pi_1$, $\zeta_\ell^{[\vec{j}]}$ must be completely annihilated by $X_{i_1}\cdots X_{i_{t-1}}$. The annihilation operators from $X_{i_1}\cdots X_{i_{t-1}}$ tasked with this must each skip over $e_{i_t}$ at a scalar cost $q_{i_t k}$ for some $k\in[N]$. Since $\zeta_\ell^{[\vec{j}]}$ is a linear combination of $e_{k_1}\otimes\cdots\otimes e_{k_d}$, $\vec{k}\in [\vec{j}]$, the total scalar cost will be $q_{i_t}(\vec{j})$. The remaining actions of $X_{i_1}\cdots X_{i_{t-1}}$ (any creation operators and any annihilation operators acting on vectors left of $e_{i_t}$ in the tensor product) are unaffected by the presence of $e_{i_t}$. In summary, the contribution to $c_t$ from the terms in the sum above is as follows:
\[
	\sum_{d\geq 0} \sum_{[\vec{j}]\in [N]^d/\sim} \sum_{\ell} q_{i_t}(\vec{j})e_{i_t} \left\lge  X_{i_1}\cdots X_{i_{t-1}} \zeta_{\ell}^{[\vec{j}]}, \Omega\right\rge_Q \left\lge X_{i_{t+1}}\cdots X_{i_n}\Omega, \zeta^{[\vec{j}]}_\ell\right\rge_Q.
\]
Noting that
\[
	\sum_{\ell}  \left\lge  X_{i_1}\cdots X_{i_{t-1}} \zeta_{\ell}^{[\vec{j}]}, \Omega\right\rge_Q \left\lge X_{i_{t+1}}\cdots X_{i_n}\Omega, \zeta^{[\vec{j}]}_\ell\right\rge_Q= \left\lge X_{i_1}\cdots X_{i_{t-1}} p_{[\vec{j}]} X_{i_{t+1}}\cdots X_{i_n}\Omega, \Omega\right\rge_Q,
\]
we see that $c_t$ has the claimed value.

Thus for $s\in[N]$ we have
	\begin{align*}
		\lge X_s, X_{i_1}\cdots X_{i_n}\rge_{\tau_Q} &= \lge e_s, \pi_1X_{i_1}\cdots X_{i_n}\Omega\rge_Q\\
			&= \sum_{t=1}^n \lge e_s, e_{i_t}\rge_Q \lge\Omega, X_{i_1}\cdots X_{i_{t-1}} \Xi_{i_t} X_{i_{t+1}}\cdots X_{i_n}\Omega\rge_Q\\
			&= \lge p_\Omega, \partial_s^{(Q)}(X_{i_1}\cdots X_{i_n})\rge_{HS}\\
			&= \lge 1\otimes 1^\circ, \partial_s^{(Q)}(X_{i_1}\cdots X_{i_n})\rge_{\tau_Q\otimes\tau_Q^{op}}.
	\end{align*}
Extending this via linearity from monomials to the dense subset $\cz\lge X_1,\ldots, X_N\rge$ in the domain of $\partial_s^{(Q)}$ concludes the proof.
\end{proof}

\begin{cor}\label{closable}
Suppose $\Xi_j\in HS(\fx_Q)$. Then
	\[
		\cz\lge X_1,\ldots, X_N\rge\otimes \cz\lge X_1,\ldots, X_N\rge^{op}\subset \Dom{\partial_j^{(Q)*}}.
	\]
In particular, for $a,b\in \cz\lge X_1,\ldots, X_N\rge$
	\begin{align}\label{adj}
		\partial_j^{(Q)*}(a\otimes b^\circ) = aX_j b - m\circ(1\otimes \tau_Q\otimes 1)\circ(1\otimes \partial_j^{(Q)} + \partial_j^{(Q)}\otimes 1)(a\otimes b^\circ),
	\end{align}
where $m(a\otimes b^\circ)=ab$. Consequently, $\partial_j^{(Q)}$ is closable.
\end{cor}
\begin{proof}
The formula is a simple computation (\emph{cf.} Proposition 4.1 in \cite{Voi98}, the proof of Theorem 34 in \cite{Dab}, or Corollary 2.4 in \cite{Ne15}). The closability of $\partial_j^{(Q)}$ then follows because this formula holds on the dense subset $\cz\lge X_1,\ldots, X_N\rge\otimes \cz\lge X_1,\ldots, X_N\rge^{op} \subset L^2(\Ga_Q\bar{\otimes}\Ga_Q^{op},\tau_Q\otimes \tau_Q^{op})$.
\end{proof}

Let us update the notation $\partial_j^{(Q)}$ so that from now on it denotes the closure of this derivation.

Let $\phi: S_n\to \bx(H_Q^n)$ be the quasi-multiplicative function defined in \cite{BS94} and define $P^{(n)}= \sum_{\si\in S_n}\phi(\si)$. According to \cite{BS94}, we have
\[
\lge \xi,\eta\rge_Q= \de_{n,m} \lge \xi, P^{(n)}\eta\rge_0, \text{ for } \xi\in H_Q^n, \eta\in H_Q^m.
\]
Here $\lge\cdot, \cdot\rge_0$ is the inner product associated to $(\Ga_0(\rz^N), \tau_0)$. Let $q=\max_{1\le i,j\le N} |q_{ij}|$. Assume $q<1$. By \cite{Boz}*{Theorem 2}, we find
\begin{equation*}\label{pnorm}
  \|(P^{(n)})^{-1}\|\le \Big[(1-q)\prod_{k=1}^\8\frac{1+q^k}{1-q^k}\Big]^{n}.
\end{equation*}
Using the Gauss identity, we have the estimate
\begin{equation}\label{pnorm2}
  \|(P^{(n)})^{-1}\|\le \Big[(1-q) \Big(\sum_{k=-\8}^\8(-1)^k q^{k^2}\Big)^{-1}\Big]^{n}\le \Big(\frac{1-q}{1-2q}\Big)^{n}.
\end{equation}

\begin{lemma}\label{power}
If $\eps>0$ and $q(3-2q+(3+\eps)^2N^2)<1$, then there exists a noncommutative power series representation of $\Xi_i$ with radius of convergence greater than $R=
\frac{2+\eps}{1-q}>\|X_i\|$ such that
\[
\|\Xi_i-1\otimes 1^\circ\|_R\le \frac{qN^2(3+\eps)^2}{1-q(3-2q+(3+\eps)^2N^2)}=:\pi(q,N)
\]
for $i=1,\ldots,N$.

\end{lemma}
\begin{proof}
Following the argument of \cite{Dab}, let $G_{n}$ denote the Gram matrix of the inner product on $(\Ga_Q,\tau_Q)$ from the natural basis $(e_{i_1}\otimes \cdots\otimes e_{i_n})$ of $H_Q^n$, where $\vec{i}\in [N]^n$. Namely, $G_n$ is the matrix of $P^{(n)}$ in the basis $(e_{i_1}\otimes \cdots\otimes e_{i_n})$. We write $\psi_{\vec{i}}=W(e_{i_1}\otimes\cdots\otimes e_{i_n})$ for the Wick word. From the isomorphism $L^2(\Ga_Q,\tau_Q)\cong \fx_Q$, we can also write
\[
(G_n)_{\vec{i}\vec{j}} = \lge e_{i_1}\otimes\cdots\otimes e_{i_n},e_{j_1}\otimes\cdots\otimes e_{j_n}\rge_Q = \lge \psi_{\vec{i}}, \psi_{\vec{j}}\rge_{\tau_Q}.
\]
Let us define inductively the noncommutative polynomials, $\psi_\eps = 1$ for the empty word $\eps$ and
\begin{equation}\label{indu}
  \psi_{i_{1},\ldots, i_n}(Y_1,\ldots,Y_N)=Y_{i_1}\psi_{i_2,\ldots,i_n}-\sum_{j= 2}^n \de_{i_1,i_j} \prod_{k=2}^{j-1} q_{i_1i_k} \psi_{i_2,\ldots,i_{j-1},i_{j+1},\ldots,i_n}(Y_1,\ldots,Y_N),
\end{equation}
where the product over empty set is understood to be 1. It can be checked that $\psi_{\vec{i}} = \psi_{\vec{i}}(X_1,\ldots,X_N)$; \emph{cf}. \cite{Kr00}. Let us define $B=G_n^{-1/2}$. Note that $B$ is a positive-definite symmetric $N^n\times N^n$ matrix and that $B_{\vec{i}\vec{j}}=0$ unless $\vec{i}\sim\vec{j}$. For each $|\vec{i}|=n$ let
\begin{equation}\label{cob}
  p_{\vec{i}}(Y_1,\ldots,Y_N) = \sum_{|\vec{j}|=n} B_{\vec{i}\vec{j}}\psi_{\vec{j}}(Y_1,\ldots,Y_N).
\end{equation}
Then $\{p_{\vec{i}}(X_1,\ldots,X_N)\Om\}_{|\vec{i}|=n}$ is an orthonormal basis of $H^n_Q$, and $\{p_{\vec{k}}(X_1,\ldots, X_N)\Om\}_{\vec{k}\in[\vec{i}]}$ is an orthonormal basis of $\fx_{[\vec{i}]}$. We want to write $\Xi_i$ as a sum of tensors. Unlike the $q_{ij}\equiv q$ case considered in \cite{Dab}, $\Xi_i$ behaves more like a multiplier instead of a projection. Consider
\[
\Xi_j(Y_1,\ldots,Y_N)= \sum_{n=0}^\8 \sum_{|\vec{i}|=n} q_j(\vec{i}) p_{\vec{i}}(Y_1,\ldots,Y_N)\otimes p_{\vec{i}}^*(Y_1,\ldots,Y_N).
\]
One can check that
\[
\Xi_j(X_1,\ldots,X_N)\psi_{\vec{i}} = q_j(\vec{i})\psi_{\vec{i}},
\]
which means that $\Xi_j$ can be identified as $\Xi_j(X_1,\ldots,X_N)$ via the isomorphism $\fx_Q \cong L^2(\Ga_Q, \tau_Q)$. By the change of basis formula \eqref{cob}, writing $w_{\vec{j}}= \psi_{\vec{j}}(Y_1,\ldots,Y_N)$, we have
\begin{align*}
	\Xi_j(Y_1,\ldots,Y_N)& = \sum_{n=0}^\8 \sum_{|\vec{i}|=n} q_{j}(\vec{i})\sum_{|\vec{j}|,|\vec{k}|=n} B_{\vec{i}\vec{j}} B_{\vec{k}\vec{i}} w_{\vec{j}}\otimes w_{\vec{k}}^*\\
  &= \sum_{n=0}^\8 \sum_{|\vec{j}|,|\vec{k}|=n} q_j(\vec{k}) (B^2)_{\vec{k}\vec{j}} w_{\vec{j}}\otimes w_{\vec{k}}^*,
\end{align*}
where we have used in the second equality that $B_{\vec{k}\vec{i}}=0$ unless $\vec{k}\sim\vec{i}$, in which case $q_j(\vec{i})=q_j(\vec{k})$. Taking the norm, we have
\[
\Big\|\sum_{|\vec{j}|,|\vec{k}|=n} q_j(\vec{k}) (B^2)_{\vec{k}\vec{j}} w_{\vec{j}}\otimes w_{\vec{k}}^*\Big\|_R \le q^n\sum_{|\vec{k}|=n} \|w_{\vec{k}}\|_R \Big\|\sum_{|\vec{j}|=n} (B^2)_{\vec{k}\vec{j}}w_{\vec{j}}\Big\|_R
\]
By \eqref{indu}, we find in the same way as the proof of \cite{Dab}*{Corollary 29} that
\[
\sup_{|\vec{i}|=n}\|w_{\vec{i}}\|_R\le \left(R+\frac1{1-q}\right)^n.
\]
Using the triangle inequality, we have
\begin{align*}
  \Big\|\sum_{|\vec{j}|=n} (B^2)_{\vec{k}\vec{j}}w_{\vec{j}}\Big\|_R &\le \sum_{|\vec{j}|=n} |(G_n^{-1})_{\vec{k}\vec{j}}| \sup_{|\vec{i}|=n} \|w_{\vec{i}}\|_R\le  N^{n} \|(P^{(n)})^{-1}\|\left(R+\frac1{1-q}\right)^n.
\end{align*}
Combining with \eqref{pnorm2}, we have
\[
\Big\|\sum_{|\vec{j}|,|\vec{k}|=n} q_j(\vec{k}) (B^2)_{\vec{k}\vec{j}} w_{\vec{j}}\otimes w_{\vec{k}}^*\Big\|_R \le  q^nN^{2n}\left(R+\frac1{1-q}\right)^{2n} \Big(\frac{1-q}{1-2q}\Big)^{n}.
\]
Plugging in $R=\frac{2+\eps}{1-q}$ and summing over all $n\ge 1$, we complete the proof.
\end{proof}

\section{Proof of the Main Theorem}

Let us write $\ax=\cz\lge Y_1,\ldots,Y_N\rge$. Suppose $\Xi_i\in \Gamma_Q\bar\otimes \Gamma_Q^{op}$ and is invertible in this algebra. Let $\partial_j: \ax\to  \ax \otimes  \ax^{op}$ denote the $j$-th free difference quotient with the property $\partial_j P =\sum_{P=AY_j B} A\otimes B$ for a monomial $P\in \ax$. Since $\partial_j X_i = \de_{i,j}\Xi_j\# \Xi_j^{-1}=\de_{i,j}1\otimes 1^\circ$, we have $\partial_j = \partial_j^{(Q)}\# \Xi_j^{-1}$, where $\#$ is the multiplication in $\Ga_Q\bar\otimes \Ga_Q^{op}$.

\begin{prop}\label{key}
Assume $\pi(q,N)<1$. Then we have:
\begin{enumerate}
\item[(i)] There exist noncommutative power series $\xi_j(Y_1,\ldots, Y_N)$ of convergence radius $R=\frac{2+\eps}{1-q}>\|X_i\|$ such that $\{\xi_j(X_1,\ldots,X_N)\}_{j=1}^N$ are the conjugate variables of $X_1,\ldots,X_N$.
\item[(ii)] There exists a self-adjoint potential $V(Y_1,\ldots, Y_N)$ which is also a noncommutative power series of convergence radius $R$ such that $D_i V(Y_1,\ldots, Y_N) = \xi_i(Y_1,\ldots, Y_N)$ where $D_i$ is the cyclic gradient, i.e., $D_i P = \sum_{P=AY_i B} BA$ for $P\in \ax$.
\item[(iii)] $\lim_{q\to 0} \|\xi_i(Y_1,\ldots, Y_N)- Y_i\|_R = 0$ for $i=1,\ldots,N$.
\end{enumerate}
\end{prop}
\begin{proof}
By Lemma \ref{power}, $\Xi_j^{-1}=\Xi_j^{-1}(X_1,\ldots, X_N)$ for a noncommutative power series $\Xi_j^{-1}(Y_1,\ldots,Y_N)$ and we can define a noncommutative power series
	\begin{align*}
		\xi_j(Y_1,\ldots, Y_N):= (\Xi_j^{-1})^*&(Y_1,\ldots, Y_N)\# Y_j\\
			&- m\circ(1\otimes\tau_Q\otimes 1)\circ(1\otimes \partial_j^{(Q)}+\partial_j^{(Q)}\otimes 1)((\Xi_j^{-1})^*(Y_1,\ldots, Y_N))\in \ax,
	\end{align*}
where $(a\otimes b^\circ)\# x = axb$ and $m(a\otimes b^\circ)=ab$. Then by (\ref{adj}) we have
	\begin{align*}
		\xi_j:=\xi_j(X_1,\ldots, X_N) = \partial_j^{(Q)*}((\Xi_j^{-1})^*).
	\end{align*}
Consequently for $P\in \cz\lge X_1,\ldots, X_N\rge$ we have
	\begin{align*}
		\lge \xi_j, P\rge_{\tau_Q} = \lge \Xi_j^{-1}, \partial_j^{(Q)}(P)\rge_{HS}= \lge 1\otimes 1^\circ, \partial_j(P)\rge_{\tau_Q\otimes\tau_Q^{op}};
	\end{align*}
that is, $\xi_j$ is a conjugate variable.

Let $\nx$ be the number operator acting on $\cz\lge Y_1,\ldots,Y_N\rge$; that is, $\nx$ is defined by $\nx P = d P$ for any monomial $P$ of degree $d$. Let $\Sigma$ denote the inverse of $\nx$ restricted to polynomials with no degree zero term. Define
	\begin{align*}
		V(Y_1,\ldots, Y_N)=\Sigma\left( \frac{1}{2}\sum_{i=1}^N \xi_i(Y_1,\ldots, Y_N) Y_i + Y_i \xi_i(Y_1,\ldots, Y_N)\right).
	\end{align*}
Then by precisely the same arguments as in Step 4 of the proof of Theorem 34 in \cite{Dab}, one can see that $D_iV(Y_1,\ldots, Y_N)=\xi_i(Y_1,\ldots, Y_N)$. Indeed, thanks to Proposition \ref{deriv} and part (i) above, Lemma 36 in \cite{Dab} can be verified using Lemma 12 in \cite{Dab} in our setting. The rest argument of Step 4 is algebraic, and does not use our particular inner product of $\fx_Q$.

Finally, Lemma \ref{power} implies that $\Xi_j^{-1}(Y_1,\ldots, Y_N)$ converges to $1\otimes 1^{\circ}$ with respect to the $R$-norm as $q\to 0$. By an argument similar to that of Lemma 4.3 in \cite{Ne15}, it is easy to see that this implies $\lim_{q\to 0} \|\xi_j(Y_1,\ldots, Y_N) - Y_j\|_R =0$.
\end{proof}

\begin{proof}[Proof of Theorem \ref{main}]
This follows from Proposition \ref{key} and the free monotone transport result of Guionnet and Shlyakhtenko \cite{GS14}*{Corollary 4.3}.
\end{proof}

\section*{Acknowledgements}
B.N. would like to thank Dimitri Shlyakhtenko for his comments about the paper, and is grateful for the support from the UCLA Dissertation Year Fellowship and the NSF Mathematical Sciences Postdoctoral Research Fellowship. Q.Z. would like to thank Michael Brannan, Alice Guionnet, and Marius Junge for helpful conversations. He also thanks the financial support from Prof. Horng-Tzer Yau and the Center of Mathematical Sciences and Applications at Harvard University. Both authors would like to thank NCGOA 2015 for providing the occasion for them to collaborate.

\bibliographystyle{alpha}
\bibliography{qisom}
\end{document}